\documentclass[12pt,a4paper,a4wide]{amsart}
\usepackage{amssymb,amsmath,amsfonts,amscd,epsfig}

\usepackage[centering,paperwidth=21.40cm,paperheight=29.7cm,body={15.06true cm,24.7true cm}]{geometry}

\newtheorem{theorem}{Theorem}[section]

\newtheorem{lemma}[theorem]{Lemma}

\theoremstyle{definition}
\newtheorem{definition}[theorem]{Definition}
\theoremstyle{remark}
\newtheorem{rem}[theorem]{Remark}
\numberwithin{equation}{section}

\numberwithin{equation}{section}
\allowdisplaybreaks

\begin{document}

\title[compactness]
 {The compactness of  commutators of Calder\'{o}n-Zgymund operators with Dini condition}

\author{Meng Qu and Ying Li}
\address
        {Meng Qu and Ying Li\\
                School of Mathematical and Computer Sciences \\
                Anhui Normal University \\
                Wuhu 241003 \\
                China}
        \email{  qumeng@mail.ahnu.edu.cn; ll@mail.ahnu.edu.cn}

 \thanks{Supported by NNSF of China (No.\,11471033),
                NSF of Anhui Provincial (No.\,1408085MA01),
                 and University NSR Project of Anhui Province (No.\,KJ2014A087).}


\date{Oct 25, 2016 
}


\begin{abstract}
        Let $T$ be the  $\theta$-type Calder\'{o}n-Zgymund operator with Dini condition.
        In this paper, we prove that  for $b\in {\rm  CMO}(\mathbb R^n)$,  the commutator generated
         by $T$ with $b$ and the corresponding maximal commutator, are both  compact  operators on  $L^{p}(\omega)$ spaces,
         where $\omega$ be the Muchenhoupt $A_p$ weight function and $1<p<\infty$.
\end{abstract}

\maketitle

\section{Introduction and main results}

Let $\theta:[0,1]\rightarrow[0,\infty)$ be a continuous, increasing, subadditive function with $\theta(0)=0$, we say that $\theta$ satisfies the $Dini$ condition if $\int_{0}^{1}\theta(t)\frac{dt}{t}<\infty.$
A measurable function $K(\cdot,\cdot)$ on $\mathbb R^n\times\mathbb R^n\backslash\{(x,x):x\in\mathbb R^n\}$ is said to be a $\theta$-type kernel if it satisfies
\begin{enumerate}
\item $|K(x,y)|\leq \frac{C}{|x-y|^n}\quad\text{whenever }x\neq y ,$
\item $|K(x,y)-K(x',y)|+|K(y,x)-K(y,x')|\leq \theta(\frac{|x-x'|}{|x-y|})\frac{1}{|x-y|^{n}}\quad\text{whenever }|x-y|\geq2|x-x'|.$
\end{enumerate}
We say that $T$ is a $\theta$-type Calder\'on--Zygmund operator if
\begin{enumerate}

\item  $T$ can be extended to be a bounded linear operator on $L^2(\mathbb R^n);$

\item  There is a $\theta$-type kernel $K(x,y)$ such that
\begin{equation}
Tf(x):=\int_{\mathbb R^n}K(x,y)f(y)\,dy
\end{equation}
\end{enumerate}
for all $f\in C^\infty_0(\mathbb R^n)$ and for all $x\notin {\rm supp}\,f$, where $C^\infty_0(\mathbb R^n)$ is the space consisting of all infinitely differentiable functions on $\mathbb R^n$ with compact supports. Historically, The $\theta$-type Calder\'on--Zygmund operator was introduced by Yabuta in  \cite{yabuta} as a nature generalization of the classical  Calder\'on-Zygmund operator. We note that  when $\theta(t)=t^{\delta}$ with $0<\delta\leq1$, the $\theta$-type  operator is just the classical Calder\'on-Zygmund operator with standard kernel (see \cite{duoand,garcia}).
Given a locally integrable function $b$ defined on $\mathbb R^n$, and given a $\theta$-type Calder\'on--Zygmund operator $T $, the linear commutator $[b,T]$ generated by $b$ and $T $ is defined for smooth, compactly supported function $f$ as
\begin{equation}\label{eq:bT}
\begin{split}
[b,T ]f(x)&:=b(x)Tf(x)-T(bf)(x)\\
&=\int_{\mathbb R^n}\big[b(x)-b(y)\big]K(x,y)f(y)\,dy.
\end{split}
\end{equation}
Also, we can define the maximal commutator of $\theta$-type Calder\'on--Zygmund operator as
$$
[b,T^{\ast}] f(x)=\sup_{\epsilon>0}\left|T_{\epsilon}f(x)\right|,
$$
where $
T_{\epsilon}f(x)=\int_{|x-y|\geq\epsilon}\big[b(x)-b(y)\big]K(x,y)f(y){d}y
$ be the truncated part of \eqref{eq:bT}. Historically, commutators related to singular integral  gave a new characterization of $\rm BMO$ function space, see Coifman, Rochberg and Weiss \cite{coifman-rochberg-weiss-1976} and Janson\cite{Janson-1978}. Recently,  Lerner \cite{Lerner} considered the weighted $L^p(\omega)$ estimates for $T$ with the sharp norm constant with respect to weight function, where $\omega$ be Muchenhoupt weight function and $1<p<\infty$(see also Quak-Yang \cite{quek} for boundedness  without the sharp constant case). As the consequence of \cite{Alvarez}, $[b,T]$ is bounded also on  $L^p(\omega)$.

On the other hand, many researcher were interested  in discussing the compact of commutators.
 Uchiyama \cite{uchiyama-1978} proved
that  the commutator generated by a locally integral function $b$ with the homogenous singular integral(Lipschitz kernel) is compact on $L^p$
 if and only if $b\in VMO(\mathbb R^n)$.
  Recently, Chen, Ding, Hu etc. consider the compactness of commutators generated by singular integrals with  rougher kernel, see\cite{chen-2015}.
 Torres etc discuss the multilinear case for compactness.
Krantz and Li\cite{KraLi2} disscuss the compact on $L^p(X)$, where $X$ be he space of homogenous type. applicaiotn.
Recently to study the regularity of solutions to the Beltrami equation, Clop and Cruz\cite{Clop-Cruz-2013-Ann.Fenn.Math} proved that when $b$ is a VMO function, the commutator for stand Calderon-Zygmund operator is compact on $L^p(\omega)$. It is nature to ask whether $b$ belongs to  VMO is also the sufficient condition for the compactness for  the commutator generated by the $\theta$-Calderon-Zygmund operator on  $L^p(\omega)$. This note will give a formative answer,
moreover, we prove that maximal commutator $[b,T^{\ast}] $ share the same result. More precisely, we give here the main result as the following theorem.
\begin{theorem}\label{thm:compact}
Let $T$ be a $\theta$-Calder\'on-Zygmund operator with $\theta$ satisfying the $Dini$ condition and $\omega\in A_{p}$ with $1<p<\infty$. If  $b\in VMO(\mathbb{R}^{n})$, then
 \begin{enumerate}
   \item $[b,T]\colon L^{p}(\omega)\to L^{p}(\omega)$ is compact.
   \item $[b,T^*]\colon L^{p}(\omega)\to L^{p}(\omega)$ is compact.
 \end{enumerate}
 \end{theorem}

\begin{rem}
  For the compact of  $[b,T]$, our proof is quiet similar as in \cite{Clop-Cruz-2013-Ann.Fenn.Math}, based on the argument
  in\cite{KraLi2}. However we do some modification by choosing a suitable smooth truncation
  technique, the idea is coming from B\'enyi-Dami\;anMoen-Torres \cite{Benyi-Damian-Moen-Torres-2015} dealing sharp constant for $A_p$ estimate.
  To some extend, our proof here simplify the one in  \cite{Clop-Cruz-2013-Ann.Fenn.Math}. On the other hand, to obtain
  the compact of  $[b,T^*]$, we need the boundedness of  $[b,T^*]$ on $L^p(\omega)$. We mention that
  since   $[b,T^*]$ is non-linear, it is not the direct consequences of \cite{Alvarez}. We prove it in a rather
  simple way based on some results in \cite{huguoen-lin-yang-2008-aaa}.
\end{rem}
This note is organized as following way. In Section 2, we give some definitions and some lemmas.
We deal with  $[b,T]$ in Section 3,  while $[b,T^*]$ is in  Section 4. last but not least,
we denote $s^{\prime}=\frac{s}{s-1}$ and
$C$ be a positive constant whose value may change at each
appearance.

\section{Some definitions and  technical lemmas}

%

As usually, we denote $\langle f\rangle_{E}=\frac{1}{|E|}\int_{E} f(x){d}x$. We say $\omega$ is a  weight  function if $\omega\in L_{loc}^{1}(\mathbb{R}^{n})$  such that $\omega(x)>0$ almost everywhere. A weight function $\omega$ is said to belong to the Muckenhoupt class $A_p$, $1<p<\infty$, if
$$
[\omega]_{A_{p}}:=\sup_{Q} \langle\omega\rangle_Q \langle\omega^{-\frac{p'}{p}}\rangle_Q^{\frac{p}{p'}}<\infty,$$
where the supremum is taken over all cubes $Q\subset\mathbb{R}^{n}$, and where $\frac{1}{p}+\frac{1}{p'}=1$. By $L^p(\omega)$ we denote the set of measurable functions $f$ that satisfy
\begin{equation}
\|f\|_{L^p(\omega)}=\left(\int_{\mathbb{R}^{n}}|f(x)|^{p}\omega(x)\mathrm{d}x\right)^{\frac{1}{p}}<\infty.\label{eq:tat}
\end{equation}
The quantity  $\|f\|_{L^{p}(\omega)}$ defines a complete norm in  $L^{p}(\omega)$.

\begin{definition}
Suppose that $f\in L_{loc}^{1}(\mathbb{R}^{n})$ and $B\in\mathbb{R}^{n}$ is a ball.  For $a>0$, let
$$\mathcal{M}(f,B)=\frac{1}{|B|}\int_{B}|f(x)-\langle f\rangle_{B}|dx~\mbox{for any ball}\  B\subset\mathbb{R}^{n},$$
and $\mathcal{M}_{a}(f)=\sup_{|B|=a}\mathcal{M}(f,B).$
We say
\begin{enumerate}
  \item the function $f$ is said to belong to $BMO(\mathbb{R}^{n})$ if there exists a constant $C>0$ such that $\|f\|_{BMO}:=sup_{a>0}\mathcal{M}_{a}\leq C$, and
  \item a function $f$  is said to belong to $VMO(\mathbb{R}^{n})$ if  $f\in BMO(\mathbb{R}^{n})$,
$$\lim_{a\rightarrow0}\mathcal{M}_{a}(f)=0 \ \mbox{and}  \lim_{a\rightarrow +\infty}\mathcal{M}_{a}(f)=0.$$
\end{enumerate}
\end{definition}

\begin{rem}\label{rem2.2}
  $VMO$ space concides with $CMO$ space, where $CMO$ space denotes the closure of $C_{c}^{\infty}$ in the $BMO$ topology.

\end{rem}

We need  the following sufficient condition for compactness in $L^{p}(\omega)$, $\omega\in A_p$ and $1<p<\infty$.
This lemma  was established in \cite[Theorem 5]{Clop-Cruz-2013-Ann.Fenn.Math}.

%
%
%

\begin{lemma}\label{lem:FreKol}
\cite{Clop-Cruz-2013-Ann.Fenn.Math}Let $p\in(1,\infty)$, $\omega\in A_{p}$, and let $\mathfrak{F\subset}L^{p}(\omega)$. Then $\mathfrak{F}$ is totally bounded if it satisfies the next three conditions:
\begin{enumerate}
\item $\mathfrak{F}$ is uniformly bounded, i.e. $\sup_{f\in\mathfrak{F}}\|f\|_{L^p(\omega)}<\infty$. \label{unif}
\item $\mathfrak{F}$ is uniformly equicontinuous, i.e.\label{equic}
$\sup_{f\in\mathfrak{F}}\|f(\cdot+h)-f(\cdot)\|_{L^{p}(\omega)}\xrightarrow{h\to0}0.$
\item $\mathfrak{F}$ uniformly vanishes at infinity,\label{decae} i.e. $\sup_{f\in\mathfrak{F}}\|f-\chi_{Q(0,R)}f\|_{L^{p}(\omega)}\xrightarrow{R\to\infty}0$, where $Q(0,R)$ is the cube with center at the origin and sidelength $R$.
\end{enumerate}
\end{lemma}
%
Technically, by Remark \ref{rem2.2}, we can approximate
VMO function by ${C}_{0}^{\infty}(\mathbb{R}^{n})$ function. More precisely, we have the following lemma.

\begin{lemma}\label{lem:sl}
For any $b\in {\rm VMO}(\mathbb R^n)$, we can approximate the function $b$ by functions $b_{j}\in {  C}_{c}^{\infty}(\mathbb{R}^{n})$ in the $BMO$ norm, such that the following is satisfying
\begin{equation}\label{eq:bj}
\|[b,T]f-[b_{j}, T]f\|_{L^{p}(\omega)}\to 0,\hspace{1cm}\text{ as } j\to \infty.
\end{equation}
\end{lemma}

%
Suppose that $\psi: [0,\infty)\to [0,1]$ satisfy (a)${\rm supp}\, \psi=\{t:  t  \geq\frac{1}{2}\}$, (b) $\psi(t)=1$ when $ t>1$ and (c)$| \psi'|\leq C$, when $\frac{1}{2}<t<1$. Then for every $\eta>0$ small enough, let us take a continuous function $K^\eta$ defined on $\mathbb{R}^{n}\times\mathbb{R}^{n}$ as
$$
K^{\eta}(x,y)=K(x,y)\psi({| x-y|}/{\eta}).
$$
  We can find that $K^{\eta}(x,y)$ satisfy
\begin{enumerate}
\item $K^{\eta}(x,y)=K(x,y)$,   if $|x-y|\geq\eta,$
\item $|K^{\eta}(x,y)|\lesssim\frac{1}{|x-y|^{n}}$,   if $\frac{\eta}{2}<|x-y|<\eta,$
\label{enu:-para-2}
\item $K^{\eta}(x,y)=0$ ,   if $|x-y|\leq\frac{\eta}{2}.$
\end{enumerate}
 Now, we denote
$$T^{\eta}f(x)=\int_{\mathbb{R}^{n}}K^{\eta}(x,y)f(y)\mathrm{d}y$$
and
$$[b, T^{\eta}]f(x)=\int_{\mathbb{R}^{n}}(b(x)-b(y))K^{\eta}(x,y)f(y)\mathrm{d}y.$$
The following lemma shows that if $b\in{ C}^1_c(\mathbb R^n)$, the commutators $[b, T^{\eta}]$ approximate $[b, T]$ in the operator norm,
this is quiet the same in \cite[Lemma 7]{Clop-Cruz-2013-Ann.Fenn.Math}. Since the proof is not relies  on the kernel condition,
the proof is also the same, we omit the detail.
\begin{lemma}\label{lema2}
Let $b\in{  C}^1_c(\mathbb R^n)$. There exists a constant $C=C(n,C_0)$ such that
$$
|[b, T ]f(x)-[b, T^{\eta}] f(x)|\leq C\eta|| b||_\infty  Mf(x) \hspace{1cm}x\in \mathbb{R}^{n} \text{a.e.,}
$$
for every $\eta>0$. As a consequence,
$$
\lim_{\eta\to 0}\parallel [b, T^{\eta}] - [b, T ]\parallel_{L^p(\omega)\to L^p(\omega)}= 0,
$$
whenever $\omega\in A_p$ and $1<p<\infty$.
\end{lemma}

\section{Proof of  part (1) in  Theorem \ref{thm:compact}}

%
We mention that $[b, T]$ is bounded on $L^p(\omega)$ for $\omega\in L^p(\omega) (1<p<\infty)$ in Section 1. Now we denote
$$\mathfrak{F}=\{[b,T^{\eta}]f; f\in L^p(\omega), \parallel f\parallel_{L^p(\omega)}\leq 1, b\in C_{c}^{\infty}(\mathbb R^{n}) \}.$$
Thanks to Lemma 2.2 and 2.3, to prove  part (1) of  Theorem \ref{thm:compact}, it is suffice to check that $\mathfrak{F}$ satisfying the
 condition (2) and (3) in Lemma 2.1.
This is to say, we need to prove the following two equations,
\begin{equation}\label{eq:key1}
\lim_{h\to 0}\sup_{f\in\mathfrak{F}}\parallel [b,T^{\eta}]f(\cdot)-[b,T^{\eta}]f(\cdot+h)\parallel_{L^{p}(\omega)}=0,
\end{equation}
and
\begin{equation}\label{eq:key2}
 \lim_{R\to \infty}\sup_{f\in\mathfrak{F}}\ \left(\int_{|x|>R }|[b,T^{\eta}]f(x)|^{p}\omega(x)\mathrm{d}x\right)^\frac1p=0.
\end{equation}

To prove \eqref{eq:key1}.  Indeed, for $b\in C_{c}^{\infty}(\mathbb R^{n})$
\begin{align*}
[b,T^{\eta}]&f(x)-[b,T^{\eta}]f(x+h) \\
&=b(x)T^{\eta}f(x)-T^{\eta}(bf)(x)-b(x+h)T^{\eta}f(x+h)+T^{\eta}(bf)(x+h)\\
&=b(x)T^{\eta}f(x)-T^{\eta}(bf)(x)-b(x+h)T^{\eta}f(x+h)\\
&\quad \quad +T^{\eta}(bf)(x+h)-b(x+h)T^{\eta}f(x)+b(x+h)T^{\eta}f(x)\\
&=:A(x,h)+B(x,h),
\end{align*}
where
\begin{align*}
A(x,h)&=b(x)T^{\eta}f(x)-b(x+h)T^{\eta}f(x)\\
&=(b(x)-b(x+h))\int_{\mathbb{R}^{n}}K^{\eta}(x,y)f(y)\mathrm{d}y
\end{align*}
and
\begin{align*}
B(x,h)&=b(x+h)T^{\eta}f(x)-T^{\eta}(bf)(x)-b(x+h)T^{\eta}f(x+h)+T^{\eta}(bf)(x+h)\\
&=\int_{\mathbb{R}^{n}}(b(x+h)-b(y))(K^{\eta}(x,y)-K^{\eta}(x+h,y))f(y)\mathrm{d}y\\
&=\int_{\mathbb{R}^{n}}(b(x+h)-b(y))\left[K(x,y)\psi\left(\frac{| x-y|}{\eta}\right)-K(x+h,y)\psi\left(\frac{| x+h-y|}{\eta}\right)\right]f(y)\mathrm{d}y.
\end{align*}
For $A(x,h)$, using the regularity of the function $b$ and the definition of the operator $T^{\ast}$,
\begin{align*}
|A(x,h)|
& \leq \|\nabla b\|_{\infty}|h|\left |\int_{|x-y|>\frac{\eta}{2}}(K^{\eta}(x,y)-K(x,y))f(y)\mathrm{d}y+\int_{|x-y|>\frac{\eta}{2}}K(x,y)f(y)\mathrm{d}y\right |\\
&\leq\|\nabla b\|_{\infty}|h|\left(\int_{|x-y|>\frac{\eta}{2}}| K^{\eta}(x,y)-K(x,y)|| f(y)|\mathrm{d}y+T^{\ast}f(x)\right)\\
&\leq \|\nabla b\|_{\infty}|h|\left(\int_{\eta\geq |x-y|>\frac{\eta}{2}}\frac{|f(y)|}{|x-y|^n}\mathrm{d}y+T^{\ast}f(x)\right)\\
& \leq \|\nabla b\|_{\infty}| h|(C\,Mf(x)+T^{\ast} f(x))
\end{align*}
for some constant $C>0$ independent of $h$. Therefore
 \begin{equation}
(\int|A(x,h)|^p\omega(x)\mathrm{d}x)\leq
C|h|\|f\|_{L^{p}(\omega)},\label{eq:comp1}\end{equation}
for $C$ independent of $f$ and $h$. Here we used the boundedness of $M$ and $T^{\ast}$ on $L^p(\omega)$ (see \cite{Hytonen-Roncal-Tapoila-2017-israelMJ})

Suppose $|h|<\frac{\eta}{4}$, then
$$
|B(x,h)|\leq|B_1(x,h)|+|B_2(x,h)|,
$$
where
$$|B_1(x,h)|=\left|\int_{\mathbb{R}^{n}}(b(x+h)-b(y))(K(x,y)-K(x+h,y))\psi(\frac{| x-y|}{\eta})f(y)\mathrm{d}y\right|$$
and
$$|B_2(x,h)|=\left|\int_{\mathbb{R}^{n}}(b(x+h)-b(y))K(x+h,y)(\psi(\frac{| x-y|}{\eta})-\psi(\frac{| x+h-y|}{\eta}))f(y)\mathrm{d}y\right|.$$
For $|B_1(x,h)|$, we have
\begin{align*}
|B_1(x,h)|&\leq\int_{| x-y|>\frac{\eta}{2}}| b(x+h)-b(y)|| K(x,y)-K(x+h,y)| | f(y)|\mathrm{d}y\\
&\leq C\| b\|_{\infty}\int_{| x-y|>\frac{\eta}{2}}\theta\left(\frac{| h|}{| x-y|}\right)\frac{1}{| x-y|^{n}}| f(y)|\mathrm{d}y\\
&= C\|b\|_{\infty}\int_{| y|>\frac{\eta}{2}}\theta\left(\frac{| h|}{| y|}\right)\frac{1}{| y|^{n}}| f(x-y)|\mathrm{d}y\\
&= C\|b\|_{\infty}\sum_{k=1}^{\infty}\int_{\frac{\eta}{2}2^{k-1}<| y|<\frac{\eta}{2}2^{k}}\theta\left(\frac{| h|}{| y|}\right)\frac{1}{| y|^{n}}| f(x-y)|\mathrm{d}y\\
&\leq  C\|b\|_{\infty}\sum_{k=1}^{\infty}\theta\left(\frac{| h|}{\frac{\eta}{2}2^{k-1}}\right)\frac{1}{(\frac{\eta}{2}2^{k-1})^{n}}\int_{| y|<\frac{\eta}{2}2^{k}}| f(x-y)|\mathrm{d}y\\
&= C\|b\|_{\infty}\sum_{k=1}^{\infty}\theta\left(\frac{| h|}{\frac{\eta}{2}2^{k-1}}\right)\frac{2^{n}}{(\frac{\eta}{2}2^{k})^{n}}\int_{| y|<\frac{\eta}{2}2^{k}}| f(x-y)|\mathrm{d}y\\
&= C2^{n}\|b\|_{\infty}Mf(x)\sum_{k=1}^{\infty}\theta\left(\frac{4| h|}{\eta\cdot2^{k}}\right)\\
&\leq C\|b\|_{\infty}Mf(x)\int_{0}^{\frac{4|h|}{\eta}}\theta(t)\mathrm{d}t.
\end{align*}
So, we can get
\begin{align*}
(\int_{R^n}|B_1(x,h)|^{p}\omega(x)\mathrm{d}x)^{\frac{1}{p}}\leq  C\| b\|_{\infty}\| f\|_{L^{p}(\omega)}\int_{0}^{\frac{4| h|}{\eta}}\theta(t)\frac{1}{t}\mathrm{d}t
\end{align*}
For $|B_2(x,h)|$, we know
\begin{align*}
|B_2(x,h)|&\leq|\int_{| x-y|>\frac{\eta}{4}}(b(x+h)-b(y))K(x+h,y)f(y)\mathrm{d}y|\\
&\leq C\| \nabla b\|_{\infty}\int_{| x-y|>\frac{\eta}{4}}\frac{1}{| x+h-y|^{n-1}}| f(y)|\mathrm{d}y\\
&\leq C\| \nabla b\|_{\infty}\sum_{j=1}^{+\infty}\int_{2^{-j-1}\eta<| x-y|<2^{-j}\eta}\frac{1}{| x+h-y|^{n-1}}| f(y)|\mathrm{d}y\\
&\leq C\| \nabla b\|_{\infty}\sum_{j=1}^{+\infty}\frac{1}{(2^{-j}\eta)^{n-1}}\int_{| x-y|<2^{-j}\eta}| f(y)|\mathrm{d}y\\
&\leq C\| \nabla b\|_{\infty}\sum_{j=1}^{+\infty}2^{-j}\eta\frac{1}{(2^{-j}\eta)^{n}}\int_{| x-y|<2^{-j}\eta}| f(y)|\mathrm{d}y\\
&\leq C\| \nabla b\|_{\infty}M(f)(x)\sum_{j=1}^{+\infty}2^{-j}\eta \\
&\leq C\eta\| \nabla b\|_{\infty}M(f)(x).
\end{align*}
Thus
\begin{align*}
(\int_{R^n}|B_2(x,h)|^{p}\omega(x)\mathrm{d}x)^{\frac{1}{p}}&\leq C\eta\| \nabla b\|_{\infty}(\int_{R^n} M(f)(x)^{p}\omega(x)\mathrm{d}x)^{\frac{1}{p}}\\
&\leq C\eta\|\nabla b\|_{\infty}\|f\|_{L^{p}(\omega)}
\end{align*}
$$
\lim_{h\to 0}\sup_{f\in\mathfrak{F}}\| [b,T^{\eta}]f(\cdot+h)-[b,T^{\eta}]f(\cdot)\|_{L^{p}(\omega)}=0.
$$


Finally, we show the decay at infinity of the elements of $\mathfrak{F}$. Let $x$ be such that $|x|>R>2R_0$, suppose that ${\rm supp}\,b\subset B(0,R_0)$. Then, $b\in C_{c}^{\infty}$, $x\not\in\mathrm{supp}\, b$, and
\begin{align*}
|[b,T^{\eta}]f(x)|
& =\left|\int_{\mathbb{R}^{n}}(b(x)-b(y))K^{\eta}(x,y)f(y)\mathrm{d}y\right| \\
&\leq C_{0}\|b\|_{\infty}\int_{\mathrm{supp}\, b}\frac{|f(y)|}{|x-y|^{n}}\mathrm{d}y\\
 &\leq \frac{C\| b\|_{\infty}}{|x|^{n}}\int_{\mathrm{supp}\, b} |f(y)|\,\mathrm{d}y \\
 &\leq \frac{C \|b\|_{\infty}}{|x|^{n}}\,\|f\|_{L^{p}(\omega)}\,\left(\int_{\mathrm{supp}\, b}\omega(y)^{-\frac{p'}{p}}dy\right)^{\frac{1}{p'}},
\end{align*}
whence
\begin{align}\label{decay}
\left(\int_{|x|>R }|[b,T^{\eta}]f(x)|^{p}\omega(x)\mathrm{d}x\right)^\frac1p\leq C\|b\|_\infty\|f\|_{L^{p}(\omega)}\left(\int_{|x|>R}\frac{\omega(x)}{|x|^{np}}\,\mathrm{d}x\right)^\frac{1}{p}.
\end{align}
By \cite[Lemma 2.2]{GarRub}, we have
\begin{equation}\label{eq:deca1}
\int_{|x|>R}\frac{\omega(x)}{|x|^{np}}\mathrm{d}x \leq \sum_{j=1}^{\infty}(2^{j-1}R)^{-np}(2^{j}R )^{nq}\omega(B(0,1))=\frac{C}{R^{n(p-q)}}<\infty.
\end{equation}
The right hand side above converges to $0$ as $R\to\infty$, due to \eqref{eq:deca1}.
%
%
%

Thus the proof of  part (1) in  Theorem \ref{thm:compact} follows.

\section{Proof of  part (2) in  Theorem \ref{thm:compact} }

\begin{lemma}
  Suppose that $b\in BMO$, $w\in A_p$ and $1<p<\infty$, then
  $$
  \|[b,T^*]f\|_{L^p(\omega)}\leq C \|b\|_{BMO}\|f\|_{L^p(\omega)}.
  $$
\end{lemma}
\begin{proof}
In $\mathbb{R}^{n}$ we define the unit cube, open on the right, to be the set $[0,1)^{n}$, and we let $\Omega_{0}$ be the collection of cubes in $\mathbb{R}^{n}$ which are congruent to $[0,1)^{n}$ and whose vertices lie on the lattice $Z^{n}$. If we dilate this family of cubes by a factor of $2^{-k}$ we get the collection $\Omega_{k}$, $k\in Z$; that is, $\Omega_{k}$ is the family of cubes, open on the right, whose vertices are adjacent points of the lattice $(2^{-k}Z)^{n}$. The cubes in $\bigcup_{k}\Omega_{k}$ are called dyadic cubes.
Give a function $f\in L_{loc}^{1}(\mathbb{R}^{n})$, define
$$E_{k}f(x)=\sum_{Q\in\Omega_{k}}(\frac{1}{| Q|}\int_{Q}f)\chi_{Q}(x);$$
$E_{k}f$ is the conditional expectation of $f$ with respect to the $\sigma$-algebra generated by $\Omega_{k}$. It satisfies the following fundamental identity: if $S$ is the union of cubes in $\Omega_{k}$, then
$$\int_{S}E_{k}f=\int_{S}f.$$
Define the dyadic maximal function by
$$M_{d}f(x)=\sup_{k}| E_{k}f(x)|.$$

We define the sharp maximal function by
$$M^{\sharp}f(x)=\sup_{x\in Q}\frac{1}{| Q|}\int_{Q}| f-\langle f\rangle_{Q}|,$$
where $f\in L_{loc}^{1}(\mathbb{R}^{n})$£¬ the supremum is taken over all cubes $Q$ containing x.\\
We let
$$T_{j}f(x)=\int_{| x-y|>2^{j}} K(x,y)f(y)dy,~~~~~j\in Z,$$
and $[b,T_{j}]f$ is defined as $[b,T]f$, we also define
$$[b,T^{\ast\ast}]f(x)=\sup_{j\in Z} | T_{j,b}f(x)|$$
and
$$M_{b}f(x)=\sup_{\varepsilon>0}r^{-n}\int_{| x-y|<\varepsilon}| b(x)-b(y)|| f(y)| dy$$
It is easy to  check
$$[b,T^{\ast}]f(x)\leq[b,T^{\ast\ast}]f(x)+M_{b}f(x).$$
We will show that, when $\omega\in A_{p}(\mathbb{R}^{n})$, $p\in(1,+\infty)$,
$$\|[b,T^{\ast\ast}]f\|_{L^{p}(\omega)}+\|M_{b}f\|_{L^{p}(\omega)}\leq C\|f\|_{L^{p}(\omega)}.$$
 First, we prove that
$$\|[b,T^{\ast\ast}]f\|_{L^{p}(\omega)}\leq C\parallel f\parallel_{L^{p}(\omega)}.$$
Similar to \cite[Lemma 5.15, p102-104]{duoand}, we can obtain the following Cotlar's inequality: for any $\gamma\in(0,1]$ and
any $f\in C_{c}^{\infty}$,
$$\sup|T_{j}f(x)|\leq C(M(|Tf|^{\gamma})(x)^{1/\gamma}+Mf(x)).$$
As a consequence, for $\omega\in A_{p}(\mathbb{R}^{n})$, $p\in(1,+\infty)$
$$\|\sup|T_{j}f|\|_{L^{p}(\omega)}\leq C\|f\|_{L^{p}(\omega)}.$$
So we proved that $\mathfrak{T}=\{T_{j}\}$ is bounded from $L^{p}(\omega)$ to $L^{p}(l^{\infty},\omega)$, where
$$L^{p}(l^{\infty},\omega)=\{\{f_{j}\}_{j\in Z}:\| \sup_{j\in Z} | T_{j}f|\|_{L^{p}(\omega)}<\infty\}.$$
This result combine the argument in [7], we have
$$\|[b,T^{\ast\ast}]\|_{L^{p}(\omega)}\leq C\|f\|_{L^{p}(\omega)}.$$

Next, we will prove the boundedness of $M_{b}f$, for $0<q<1$,
\begin{align*}
\| M_{b}f\|_{L^{p}(\omega)}&=\left(\int_{R^{n}}(| M_{b}f(x)|^{q})^{\frac{p}{q}}\omega(x)dx\right)^{\frac{1}{p}}\\
&\leq C\left(\int_{R^{n}}[M_{d}( (M_{b}f)^{q})]^{\frac{p}{q}}\omega(x)dx\right)^{\frac{1}{p}}\\
&\leq C\left(\int_{R^{n}}[M^{\sharp}( (M_{b}f)^{q})(x)]^{\frac{p}{q}}\omega(x)dx\right)^{\frac{1}{p}}\\
&= C\left(\int_{R^{n}}[M_{q}^{\sharp}( M_{b}f)(x)]^{p}\omega(x)dx\right)^{\frac{1}{p}}\\
&=C\| M_{q}^{\sharp}(M_{b}f)\|_{L^{p}(\omega)}.
\end{align*}
By \cite[Lemma2.3]{huguoen-lin-yang-2008-aaa} for $0<q<s<1$, we can know that
$$M_{q}^{\sharp}(\widetilde{M}_{b}f)(x)\leq C\parallel b\parallel_{\ast}[M_{s}(\widetilde{M}f)(x)+M_{L (log L)}f(x)],$$
due to  there exists some constant $C\geq1$ such that for all $x\in \mathbb{R}^{n}$
$$ C^{-1}\widetilde{M}_{b}f(x)\leq M_{b}f(x)\leq C\widetilde{M}_{b}f(x)$$
and
$$ C^{-1}\widetilde{M}f(x)\leq Mf(x)\leq C\widetilde{M}f(x),$$
where the definitions of $\widetilde{M}_{b}$ and $\widetilde{M}$ have given in \cite{huguoen-lin-yang-2008-aaa}, because of the limitation of length, no more tautology here. So we can have
$$\| M_{q}^{\sharp}(M_{b}f)\|_{L^{p}(\omega)}\leq C\| b\|_{BMO}[\| M_{s}(\widetilde{M}f)\|_{L^{p}(\omega)}+\|M_{L (log L)}f\|_{L^{p}(\omega)}].$$
In \cite{Javier Duoandikoetxea}, we know $ M_{L (log L)}\sim M^{2}$, so we have that $ M_{L (log L)}:L^{p}(\omega)\rightarrow L^{p}(\omega)$ is bounded, so we only need prove the bounded of $M_{s}(\widetilde{M}f)$, now we give the proof,
\begin{align*}
\| M_{s}(\widetilde{M}f)\|_{L^{p}(\omega)}&=(\int_{R^{n}}(M_{s}(\widetilde{M}f)(x))^{p}\omega(x)dx)^{\frac{1}{p}}\\
&=(\int_{R^{n}}(M((\widetilde{M}f)^{s})(x))^{\frac{p}{s}}\omega(x)dx)^{\frac{1}{p}}\\
&\leq C(\int_{R^{n}}(\widetilde{M}f(x))^{p}\omega(x)dx)^{\frac{1}{p}}\\
&\leq C(\int_{R^{n}}(Mf(x))^{p}\omega(x)dx)^{\frac{1}{p}}\\
&\leq C\| f\|_{L^{p}(\omega)}.
\end{align*}
Thus, we can get
$$\|M_{b}f\|_{L^{p}(\omega)}\leq C\| f\|_{L^{p}(\omega)}.$$
Whence
$$\|[b,T^{\ast}]\|_{L^{p}(\omega)}\leq C\| f\|_{L^{p}(\omega)}.$$

\end{proof}
Now, we will give the proof of $[b,T^{\ast}]$ uniformly vanishes at infinity.
\begin{lemma}
For any $b\in {\rm VMO}(\mathbb R^n)$, there exists $\{b_{j}\} \subset C_{c}^{\infty}(\mathbb{R}^{n})$ and satisfy $b=\lim_{j\rightarrow+\infty}b_{j}$ in $BMO$,  such that the following is satisfying
  $$\|[b,T^{\ast}]f-[b_{j},T^{\ast}]f\|_{L^{P}(\omega)}\rightarrow 0~~~~~as~~j\rightarrow+\infty.$$
\end{lemma}
\begin{proof}
As we know, For any $b\in VMO$ and for any $\varepsilon>0$, there exists $b_{j}\in C_{c}^{\infty}$ such that
$$\| b-b_{j}\|_{BMO}<\varepsilon.$$
It is easy to see that
\begin{align*}
| [b,T^{\ast}]f(x)-[b_{j},T^{\ast}]f(x)|&=| \sup_{\delta>0}| \int_{| x-y|>\delta}(b(x)-b(y))K(x,y)f(y)dy|\\
&-\sup_{\delta>0}| \int_{| x-y|>\delta}(b_{j}(x)-b_{j}(y))K(x,y)f(y)dy||\\
&\leq \sup_{\delta>0}|| \int_{| x-y|>\delta}(b(x)-b(y))K(x,y)f(y)dy|\\
&-|\int_{| x-y|>\delta}(b_{j}(x)-b_{j}(y))K(x,y)f(y)dy||\\
&\leq \sup_{\delta>0}| \int_{| x-y|>\delta}(b(x)-b(y))K(x,y)f(y)dy\\
&-\int_{| x-y|>\delta}(b_{j}(x)-b_{j}(y))K(x,y)f(y)dy|\\
&=[b-b_{j},T^{\ast}]f(x).
\end{align*}
So we can get
$$\| [b,T^{\ast}]f-[b_{j},T^{\ast}]f\|_{L^{p}(\omega)}\leq\|[b-b_{j},T^{\ast}]f\|_{L^{p}(\omega)}\leq C\varepsilon\| f\|_{L^{p}(\omega)}$$
and
$$\| [b,T^{\ast}]f\|_{L^{p}(\omega)}\leq\|[b_{j},T^{\ast}]f\|_{L^{p}(\omega)}+C\varepsilon\| f\|_{L^{p}(\omega)}.$$
\end{proof}
\begin{lemma}Suppose that $\omega\in A_p$ for $1<p<\infty$, then
\begin{equation}\label{eq:lim-2}
\lim_{|h|\rightarrow0}\|f(\cdot+h)-f(\cdot)\|_{L^p(w)}=0.
\end{equation}
\end{lemma}
\begin{proof}
Since $C_{c}^{\infty}(\mathbb R^n)$ is dense in $L^p(w)$ when $w\in A_p$ for $1<p<\infty$, we need only to prove that for any $f\in C_{c}^{\infty}(\mathbb{R})$,  $f$ satisfying \eqref{eq:lim-2}.
In fact  we can let ${\rm supp}\,f\subset B(0,R)$,
so ${\rm supp}\,f(\cdot+h)\subset B(0,2R)$ when $|h|$ small enough, thus
\begin{align*}
\int_{\mathbb{R}^{n}}&|f(x+h)-f(x)|^{p}\omega(x)dx\\
&=\int_{B(0,2R)}|f(x+h)-f(x)|^{p}\omega(x)dx\\
&\leq\|\nabla f\|_{\infty}|h|\omega(B(0,2R)).
\end{align*}
Since $\omega\in A_{p}$, it is obviously locally integrable, we have $\omega(B(0,2R))<\infty$, then we let $h\to 0$,   the Lemma is proved.
\end{proof}

Thus, we only need prove $[b,T^{\ast}]$ uniformly vanishes at infinity, where $b\in C_{c}^{\infty}$. we can suppose that ${\rm supp}\, b\subset\{x\in \mathbb{R}^{n} : | x|<R_{0}\}$ and $R>3R_{0}$, where $R_{0}>1$, note that, when $| x|>R$, we have

\begin{align*}
| [b,T^{\ast}]f(x)|&=\sup_{\delta>0}| \int_{| x-y|>\delta}(b(x)-b(y))K(x,y)f(y)dy|\\
&\leq C\| b\|_{\infty}\int_{| y|\leq R_{0}}\frac{| f(y)|}{| x-y|^{n}}dy\\
&=C\|b\|_{\infty}\int_{| y|\leq R_{0}}\frac{| f(y)|}{| x-y|^{n}}dy\\
&\leq C\| b\|_{\infty}\frac{1}{| x|^{n}}\int_{| y|\leq R_{0}}| f(y)| dy\\
&\leq C\| b\|_{\infty}\frac{1}{| x|^{n}}\|  f\|_{L^{p}(\omega)}\left(\int_{| y|\leq R_{0}} \omega(y)^{-\frac{p'}{p}} dy\right)^{\frac{1}{p'}}.
\end{align*}
Whence
\begin{align*}
&\left(\int_{| x|>\alpha}| [b,T^{\ast}]f(x)|^{p}\omega(x)dx\right)^{\frac{1}{p}}\\
&\leq C\| b\|_{\infty}\| f\|_{L^{p}(\omega)}\left(\int_{| y|\leq R_{0}} \omega(y)^{-\frac{p'}{p}} dy\right)^{\frac{1}{p'}}\left(\int_{| x|>R}\frac{\omega(x)}{| x|^{np}}dx\right)^{\frac{1}{p}}\\
&\leq C\| b\|_{\infty}\| f\|_{L^{p}(\omega)}\left(\int_{| x|>R}\frac{\omega(x)}{| x|^{np}}dx\right)^{\frac{1}{p}}\\
&\leq C\| b\|_{\infty}\|f\|_{L^{p}(\omega)}\frac{1}{R^{n(p-q)}}
\end{align*}
where $q<p$, so we have
$$\lim_{R\rightarrow +\infty}(\int_{| x|>R}| [b,T^{\ast}]f(x)|^{p}\omega(x)dx)^{\frac{1}{p}}=0.$$

To prove the uniform equicontinuity of $[b,T^{\ast}]$, we must see that
$$\lim_{h\rightarrow0}\sup_{f\in L^{p}(\omega)}\| [b,T^{\ast}]f(\cdot)-[b,T^{\ast}]f(\cdot+h)\|_{L^{p}(\omega)}=0.$$

In fact, for any $h\in \mathbb{R}^{n}$, we define $K_{\delta}(x,y)=K(x,y)\chi_{\{y:| x-y|>\delta\}}(y)$, so
\begin{align*}
&| [b,T^{\ast}]f(x+h)-[b,T^{\ast}]f(x)|\\
&=|\sup_{\delta>0}| \int_{| x-y|>\delta}(b(x+h)-b(y))K(x+h,y)f(y)dy|\\
&-\sup_{\delta>0}| \int_{| x-y|>\delta}(b(x)-b(y))K(x,y)f(y)dy||\\
&\leq\sup_{\delta>0}| \int_{| x-y|>\delta}(b(x+h)-b(y))K(x+h,y)f(y)dy\\
&-\int_{| x-y|>\delta}(b(x)-b^(y))K(x,y)f(y)dy|\\
&=\sup_{\delta>0}| \int_{\mathbb{R}^{n}}(b(x+h)-b(y))K_{\delta}(x+h,y)f(y)dy\\
&-\int_{\mathbb{R}^{n}}(b(x)-b(y))K_{\delta}(x,y)f(y)dy|,
\end{align*}
now, we can divided $\mathbb{R}^{n}$ into $| x-y|>\varepsilon^{-1}| h|$ and $| x-y|\leq\varepsilon^{-1}| h|$, so we can have                                                            \begin{align*}
&| [b,T^{\ast}]f(x+h)-[b,T^{\ast}]f(x)|\\
&\leq\sup_{\delta>0}[|\int_{| x-y|>\varepsilon^{-1}| h|}K_{\delta}(x,y)f(y)(b(x+h)-b(x))|\\
&+|\int_{| x-y|>\varepsilon^{-1}| h|}(K_{\delta}(x+h,y)-K_{\delta}(x,y))(b(x+h)-b(y))f(y)dy|\\
&+|\int_{| x-y|\leq\varepsilon^{-1}| h|}K_{\delta}(x,y)(b(x)-b(y))f(y)dy|\\
&+|\int_{| x-y|\leq\varepsilon^{-1}| h|}K_{\delta}(x+h,y)(b(x+h)-b(y))f(y)dy|]\\
&\leq\sup_{\delta>0}\left|\int_{| x-y|>\varepsilon^{-1}| h|}K_{\delta}(x,y)f(y)(b(x+h)-b(x))\right|\\
&+\sup_{\delta>0}\left|\int_{| x-y|>\varepsilon^{-1}| h|}(K_{\delta}(x+h,y)-K_{\delta}(x,y))(b(x+h)-b(y))f(y)dy\right|\\
&+\sup_{\delta>0}\left|\int_{| x-y|\leq\varepsilon^{-1}| h|}K_{\delta}(x,y)(b(x)-b(y))f(y)dy\right|\\
&+\sup_{\delta>0}\left|\int_{| x-y|\leq\varepsilon^{-1}| h|}K_{\delta}(x+h,y)(b(x+h)-b(y))f(y)dy\right|\\
&=E_{1}+E_{2}+E_{3}+E_{4}
\end{align*}

For $E_{1}$, we have
\begin{align*}
E_{1}&=\sup_{\delta>0}\left|\int_{| x-y|>\varepsilon^{-1}| h|}K_{\delta}(x,y)f(y)(b(x+h)-b(x))\right|\\
&\leq| h|\| \nabla b\|_{\infty}\sup_{\delta>0}\left|\int_{| x-y|\leq\varepsilon^{-1}| h|,| x-y|>\delta}K(x,y)f(y)dy\right|\\
&\leq| h| \| \nabla b\|_{\infty}T^{\ast}f(x).
\end{align*}
Thus
\begin{align*}
\| E_{1}\|_{L^{p}(\omega)}&\leq| h| \| \nabla b\|_{\infty}\| T^{\ast}f\|_{L^{p}(\omega)}\\
&\leq C| h| \| \nabla b\|_{\infty}\| f\|_{L^{p}(\omega)}.
\end{align*}
For $E_{2}$,we can know that
$$
E_{2}\leq E_{21}+E_{22},
$$
where
$$E_{21}=\sup_{\delta>0}\left|\int_{| x-y|>\varepsilon^{-1}| h|}(K(x+h,y)-K(x,y))\chi_{| x+h-y|>\delta}(y)(b(x+h)-b(y))f(y)dy\right|$$
and
$$E_{22}=\sup_{\delta>0}\left|\int_{| x-y|>\varepsilon^{-1}| h|}K(x,y)(\chi_{| x+h-y|>\delta}(y)-\chi_{| x-y|>\delta}(y))(b(x+h)-b(y))f(y)dy\right|.$$
On the one hand, we will give the estimation of $E_{21}$,
\begin{align*}
E_{21}&\leq\int_{| x-y|>\varepsilon^{-1}| h|}| K(x+h,y)-K(x,y)|| b(x+h)-b(y)| | f(y)| dy\\
&\leq C\|b\|_{\infty}\int_{| x-y|>\varepsilon^{-1}| h|} \theta(\frac{| h|}{| x-y|})\frac{1}{| x-y|^{n}}| f(y)| dy\\
&=C\|b\|_{\infty}\int_{| y|>\varepsilon^{-1}| h|} \theta(\frac{| h|}{| y|})\frac{1}{| y|^{n}}| f(x-y)| dy\\
&=C\|b\|_{\infty}\sum_{k=1}^{\infty}\int_{\varepsilon^{-1}| h|2^{k-1}<| y|<\varepsilon^{-1}| h|2^{k}} \theta(\frac{| h|}{| y|})\frac{1}{| y|^{n}}| f(x-y)| dy\\
&\leq C\|b\|_{\infty}\sum_{k=1}^{\infty}\theta(\frac{| h|}{\varepsilon^{-1}| h|2^{k-1}})(\frac{1}{\varepsilon^{-1}| h|2^{k-1}})^{n}\int_{| y|<\varepsilon^{-1}| h|2^{k}}| f(x-y)| dy\\
&\leq C\|b\|_{\infty}\sum_{k=1}^{\infty}\theta(2\varepsilon\cdot2^{-k})(\frac{2}{\varepsilon^{-1}| h|2^{k-1}})^{n}\int_{| y|<\varepsilon^{-1}| h|2^{k}}| f(x-y)| dy\\
&\leq C\|b\|_{\infty}Mf(x)\sum_{k=1}^{\infty}\theta(2\varepsilon\cdot2^{-k})\\
&\leq C\|b\|_{\infty}Mf(x)\int_{0}^{2\varepsilon}\theta(t)\frac{dt}{t},
\end{align*}
therefore
$$\| E_{21}\|_{L^{p}(\omega)}\leq C\| b\|_{\infty}\| f\|_{L^{p}(\omega)}\int_{0}^{2\varepsilon}\theta(t)\frac{dt}{t}.$$
On the other hand, for $E_{22}$,we have
\begin{align*}
E_{22}&=\sup_{\delta>0}\left|\int_{| x-y|>\varepsilon^{-1}| h|}K(x,y)(\chi_{| x+h-y|>\delta}(y)-\chi_{| x-y|>\delta}(y))(b(x+h)-b(y))f(y)dy\right|\\
&\leq\sup_{\delta>0}\left|\int_{| x-y|>\varepsilon^{-1}| h|,| x+h-y|>\delta,| x-y|<\delta}K(x,y)(b(x+h)-b(y))f(y)dy\right|\\
&+\sup_{\delta>0}\left|\int_{| x-y|>\varepsilon^{-1}| h|,| x+h-y|<\delta,| x-y|>\delta}K(x,y)(b(x+h)-b(y))f(y)dy\right|\\
&\leq E_{221}+E_{222}.
\end{align*}
Further, we are going to estimate $E_{221}$ and $E_{221}$, for $E_{221}$, when $| x-y|>\varepsilon^{-1}| h|$, $| x+h-y|>\delta$ and $0<\varepsilon<\frac{1}{4}$, then
$| x-y|>\frac{| h|}{\varepsilon}>\frac{\delta-| x-y|}{\varepsilon}$, so we have $| x-y|>\frac{\delta}{\varepsilon+1}$, and
\begin{align*}
 E_{221}&=\sup_{\delta>0}|\int_{| x-y|>\varepsilon^{-1}| h|,| x+h-y|>\delta,| x-y|<\delta}K(x,y)(b(x+h)-b(y))f(y)dy|\\
 &\leq C\parallel b\parallel_{\infty}\sup_{\delta>0}\int_{\frac{\delta}{\varepsilon+1}<| x-y|<\delta}\frac{| f(y)|}{| x-y|^{n}}dy\\
 &\leq C\parallel b\parallel_{\infty}\sup_{\delta>0}\int_{\frac{\delta}{\varepsilon+1}<| y|<\delta}\frac{| f(x-y)|}{| y|^{n}}dy\\
 &\leq C\parallel b\parallel_{\infty}\sup_{\delta>0}(\int_{\frac{\delta}{\varepsilon+1}<| y|<\delta}\frac{| f(x-y)|^{r}}{| y|^{n}}dy )^{\frac{1}{r}}\times\sup_{\delta>0}(\int_{\frac{\delta}{\varepsilon+1}<| y|<\delta}\frac{1}{| y|^{n}}dy)^{\frac{1}{r'}},
\end{align*}
where $1<r<p$, due to
\begin{align*}
\int_{\frac{\delta}{\varepsilon+1}<| y|<\delta}\frac{1}{| y|^{n}}dy&=\int_{S^{n-1}}\int_{\frac{\delta}{\varepsilon+1}}^{\delta}\frac{1}{r}dr d\sigma\\
&\leq C\ln(1+\varepsilon)\\
&\leq C\varepsilon
\end{align*}
and
\begin{align*}
&\sup_{\delta>0}\left(\int_{\frac{\delta}{\varepsilon+1}<| y|<\delta}\frac{| f(x-y)|^{r}}{| y|^{n}}dy \right)^{\frac{1}{r}}\\
&\leq\sup_{\delta>0}\left((1+\varepsilon)^{n}\delta^{-n} \int_{| y|<\delta}| f(x-y)|^{r}dy\right)^{\frac{1}{r}}\\
&\leq(1+\varepsilon)^{\frac{n}{r}} M(| f|^{r})(x)^{\frac{1}{r}},
\end{align*}
hence
\begin{align*}
\| E_{221}\|_{L^{p}(\omega)}& \leq C\varepsilon^{\frac{1}{r'}}(1+\varepsilon)^{\frac{n}{r}}\parallel b\parallel_{\infty}\parallel M(| f|^{r})^{\frac{1}{r}}\parallel_{L^{p}(\omega)}\\
&\leq C\varepsilon^{\frac{1}{r'}}(1+\varepsilon)^{\frac{n}{r}}\parallel b\parallel_{\infty}(\int_{R^{n}}| f|^{p}\omega(x)dx)^{\frac{1}{p}}\\
&\leq C\varepsilon^{\frac{1}{r'}}(1+\varepsilon)^{\frac{n}{r}}\parallel b\parallel_{\infty}\| f\|_{L^{p}(\omega)}.
\end{align*}
In a similar way, for $E_{222}$, when $| x-y|>\varepsilon^{-1}| h|$, $| x+h-y|<\delta$ and $0<\varepsilon<\frac{1}{4}$, then
$| x-y|<| x+h-y|+| h|<\delta+\varepsilon| x-y|$, so we have $| x-y|<\frac{\delta}{1-\varepsilon}$, and
\begin{align*}
E_{222}&=\sup_{\delta>0}\left|\int_{| x-y|>\varepsilon^{-1}| h|,| x+h-y|<\delta,| x-y|>\delta}K(x,y)(b(x+h)-b(y))f(y)dy\right|\\
 &\leq C\| b\|_{\infty}\sup_{\delta>0}\int_{\delta<| x-y|<\frac{\delta}{1-\varepsilon}}\frac{| f(y)|}{| x-y|^{n}}dy\\
 &\leq C\| b\|_{\infty}\sup_{\delta>0}\int_{\delta<| y|<\frac{\delta}{1-\varepsilon}}\frac{| f(x-y)|}{| y|^{n}}dy\\
 &\leq C\| b\|_{\infty}\sup_{\delta>0}(\int_{\delta<| y|<\frac{\delta}{1-\varepsilon}}\frac{| f(x-y)|^{r}}{| y|^{n}}dy )^{\frac{1}{r}}\times\sup_{\delta>0}(\int_{\delta<| y|<\frac{\delta}{1-\varepsilon}}\frac{1}{| y|^{n}}dy)^{\frac{1}{r'}},
\end{align*}
due to
\begin{align*}
\int_{\delta<| y|<\frac{\delta}{1-\varepsilon}}\frac{1}{| y|^{n}}dy&=\int_{S^{n-1}}\int_{\delta}^{\frac{\delta}{1-\varepsilon}}\frac{1}{r}dr d\sigma\\
&\leq C\ln(\frac{1}{1-\varepsilon})
\end{align*}
and
\begin{align*}
&\sup_{\delta>0}(\int_{\delta<| y|<\frac{\delta}{1-\varepsilon}}\frac{| f(x-y)|^{r}}{| y|^{n}}dy )^{\frac{1}{r}}\\
&\leq\sup_{\delta>0}(\delta^{-n} \int_{| y|<\frac{\delta}{1-\varepsilon}}| f(x-y)|^{r}dy)^{\frac{1}{r}}\\
&\leq(1-\varepsilon)^{-\frac{n}{r}} M(| f|^{r})(x)^{\frac{1}{r}},
\end{align*}
hence, we can have
\begin{align*}
\| E_{222}\|_{L^{p}(\omega)}& \leq C(\ln(\frac{1}{1-\varepsilon}))^{\frac{1}{r'}}(1-\varepsilon)^{-\frac{n}{r}}\parallel b\parallel_{\infty}\parallel f\parallel_{L^{p}(\omega)}.
\end{align*}
Next, we consider the $E_{3}$, we know
\begin{align*}
E_{3}&\leq C\parallel \nabla b\parallel_{\infty}\int_{| x-y|\leq\varepsilon^{-1}| h|}| K(x,y)|| f(y)| dy\\
&\leq C\parallel \nabla b\parallel_{\infty}\int_{| x-y|\leq\varepsilon^{-1}| h|}\frac{| f(y)|}{| x-y|^{n-1}}dy\\
&\leq C\| \nabla b\|_{\infty}\int_{| y|\leq\varepsilon^{-1}| h|}\frac{| f(x-y)|}{| y|^{n-1}}dy\\
&\leq C\| \nabla b\|_{\infty}\sum_{k=1}^{\infty}\int_{\varepsilon^{-1}| h|2^{-k}\leq| y|\leq\varepsilon^{-1}| h|2^{-k+1}}\frac{| f(x-y)|}{| y|^{n-1}}dy\\
&\leq C\| \nabla b\|_{\infty}\sum_{k=1}^{\infty}\frac{\varepsilon^{-1}| h|2^{-k}}{(\varepsilon^{-1}| h|2^{-k})^{n}}\int_{| y|\leq\varepsilon^{-1}| h|2^{-k+1}}| f(x-y)|dy\\
&\leq C\| \nabla b\|_{\infty}Mf(x)\sum_{k=1}^{\infty}\varepsilon^{-1}| h|2^{-k}\\
&\leq C\| \nabla b\|_{\infty}\varepsilon^{-1}| h|Mf(x),
\end{align*}
so, we have
$$\| E_{3}\|_{L^{p}(\omega)}\leq C\varepsilon^{-1}| h|\| \nabla b\|_{\infty}\| f\|_{L^{p}(\omega)}.$$
Finally, let's give an estimate of $E_{4}$,
\begin{align*}
E_{4}&=\sup_{\delta>0}|\int_{| x-y|\leq\varepsilon^{-1}| h|}K_{\delta}(x+h,y)(b^{\varepsilon}(x+h)-b^{\varepsilon}(y))f(y)dy|\\
&\leq C\|\nabla b\|_{\infty}\int_{| x-y|\leq\varepsilon^{-1}| h|}| K(x+h,y)|| f(y)| dy\\
&\leq C\| \nabla b\|_{\infty}\int_{| x-y|\leq\varepsilon^{-1}| h|}\frac{| f(y)|}{| x+h-y|^{n-1}}dy\\
&\leq C\| \nabla b\|_{\infty}\int_{| y|\leq\varepsilon^{-1}| h|}\frac{| f(x-y)|}{| h+y|^{n-1}}dy\\
&\leq C\| \nabla b\|_{\infty}\int_{| y|\leq(\varepsilon^{-1}+1)| h|}\frac{| f(x+h-y)|}{| y|^{n-1}}dy\\
&\leq C\| \nabla b\|_{\infty}\sum_{k=1}^{\infty}\int_{(\varepsilon^{-1}+1)| h|2^{-k}\leq| y|\leq(\varepsilon^{-1}+1)| h|2^{-k+1}}\frac{| f(x+h-y)|}{| y|^{n-1}}dy\\
&\leq C\| \nabla b\|_{\infty}\sum_{k=1}^{\infty}\frac{(\varepsilon^{-1}+1)| h|2^{-k}}{((\varepsilon^{-1}+1)| h|2^{-k})^{n}}\int_{| y|\leq(\varepsilon^{-1}+1)| h|2^{-k+1}}| f(x+h-y)|dy\\
&\leq C\| \nabla b\|_{\infty}Mf(x+h)\sum_{k=1}^{\infty}(\varepsilon^{-1}+1)| h|2^{-k}\\
&\leq C\| \nabla b\|_{\infty}(\varepsilon^{-1}+1)| h|Mf(x+h).
\end{align*}
By Lemma 4.3, we can have
\begin{align*}
\| E_{4}\|_{L^{p}(\omega)}&\leq C\| \nabla b\|_{\infty}(\varepsilon^{-1}+1)| h|\|Mf(\cdot+h)\|_{L^{p}(\omega)}\\
&\leq C\| \nabla b\|_{\infty}(\varepsilon^{-1}+1)| h|\| f\|_{L^{p}(\omega)}.
\end{align*}
If let $| h|=\frac{\varepsilon^{2}}{\varepsilon^{-1}+1}$, then
$$\lim_{\varepsilon\rightarrow0}\| E_{3}\|_{L^{p}(\omega)}=\lim_{\varepsilon\rightarrow0}\| E_{4}\|_{L^{p}(\omega)}=0.$$

%
%
%
%
%

\end{document}